\documentclass[12pt, reqno, a4paper, final]{amsart}

\usepackage[english]{babel}              % Dansk orddeling osv.
\usepackage{amsmath}                     % Matematiske kommandoer
\usepackage{a4wide}
\usepackage{amssymb}                     % Matematiske symboler
\usepackage{amsthm}
\usepackage{color}
\usepackage{hyperref}
\usepackage{marvosym}                   % Se http://www.ctan.org/tex-archive/fonts/psfonts/marvosym/marvodoc.pdf
\usepackage{mathrsfs}									  % Ekstra kroellede bogstaver
\usepackage[all]{xy}										% Laekre kommutative diagrammer
\usepackage{url}												% %Inds�ttelse af www-links
\usepackage{verbatim,epsfig}
\usepackage{xfrac} % makes it possible to do slanted fractions \sfrac{}{}

\theoremstyle{plain}

\newtheorem{thm}{Theorem}[section] 
\newtheorem{cor}[thm]{Corollary}
\newtheorem{lem}[thm]{Lemma}
\newtheorem{prop}[thm]{Proposition}
\newtheorem*{thm*}{Theorem}
\newtheorem*{prop*}{Proposition}
\newtheorem*{cor*}{Corollary}

\theoremstyle{definition}
\newtheorem{defi}[thm]{Definition}

\newtheorem{ex}[thm]{Example}

\newtheorem{rem}[thm]{Remark}

\newcommand{\NN}{{\mathbb N}}
\newcommand{\ZZ}{{\mathbb Z}}

\newcommand{\RR}{{\mathbb R}}
\newcommand{\CC}{{\mathbb C}}

\newcommand{\N}{{\mathcal N}}

\newcommand{\U}{{\mathcal U}}

\newcommand{\G}{{\mathscr G}}

\newcommand{\ip}[2]{\left\langle {#1}\hspace{0.05cm}, \hspace{0.05cm}{#2}\right \rangle}

\newcommand{\varps}{{\varepsilon}}

\newcommand{\htens}{\bar{\otimes}}

\newcommand{\tens}{\otimes}

\renewcommand{\Re}{{\operatorname{Re}}}

\newcommand{\spann}{{\operatorname{span}}}

\newcommand{\To}{\longrightarrow}

\newcommand{\Aut}{\operatorname{Aut}}

\newcommand{\id}{\operatorname{id}}

\newcommand{\alg}{{\operatorname{alg}}}

\newcommand{\conv}{{\operatorname{conv}}}
\newcommand{\I}{{\operatorname{I}}}

\newcommand{\T}{{\operatorname{(T)}}}

\renewcommand{\leq}{\leqslant}
\renewcommand{\geq}{\geqslant}
\newcommand{\twoone}{{\operatorname{II}_1}}
\address{David Kyed,
Mathematisches Institut,
Georg-Au\-gust-Uni\-versi\-t{\"a}t G{\"o}t\-ting\-en,
Bunsenstra{\ss}e 3-5,
D-37073 G{\"o}ttingen, 
Germany.}
\email{kyed@uni-math.gwdg.de}
\urladdr{www.uni-math.gwdg.de/kyed}

\begin{document}
%\maketitle
\begin{center}
\textbf{\large{Uniqueness of group-measure space Cartan subalgebras}} 

%\vspace{0.5cm}
\begin{center}
\setlength{\unitlength}{1mm}
\begin{picture}(100, 10)
  \put(55, 5){\line(1, 0){64}}
  \put(48,5){\circle*{1}}
  \put(50, 5){\circle*{1.8}}
  \put(52,5){\circle*{1}}
  \put(45, 5){\line(-1, 0){64}}
\end{picture}
\end{center}

David Kyed\\
\vspace{0.2cm}
After Adrian Ioana's IHP lectures\\
\vspace{0.2cm}
May 2011
\end{center}

\vspace{0.5cm}
These lecture notes provide an account on four lectures given by Adrian Ioana at the Institut Henri Poincar{\'e} in May 2011 regarding a number of ``uniqueness of Cartan'' type results obtained in a recent series of papers \cite{popa-vaes-grp-measure-space-decomp, chifan-peterson, vaes-cohomology-and-unique-cartan}. The main goal is to give a presentation of (a slightly scaled down version of )the main result by Chifan and Peterson in \cite{chifan-peterson} which is as reasonably self contained and low tech as possible.  The notes follow Vaes' strategy  of proof, as presented in \cite{vaes-cohomology-and-unique-cartan}, but we restrict attention to the smaller class groups that appeared in Ioana's lectures. Hopefully this restriction makes the text easily accessible to readers not familiar with the deformation/rigidity techniques that are used in the proof. \\

\vspace{0.2cm}
\noindent{\emph{Disclaimer:}}
It should be stressed that the notes are purely expository and thus claiming no originality whatsoever. Furthermore, it should be made clear that the author is by no means an expert in the field of deformation/rigidity  and it is inevitable that this fact is reflected in the the text by means of naive points of view and perhaps even honest mistakes\footnote{Needless to say, any such mistakes are of course entirely my responsibility!}. It is, however, my hope that this lack of expertise is also reflected in the text in form of slightly more detailed proofs than one often finds in the literature and that this will make the results easy to digest for (other) beginners. Comments and corrections of all sorts will be greatly appreciated.  \\

\vspace{0.2cm}
\noindent{\emph{Acknowledgements:}}
Firstly, thanks are due to Vadim Alekseev and Henrik D.~Petersen; the notes basically consist of a typed up version of our discussions concerning the material covered. Thanks are of course also due to Adrian Ioana; firstly for giving a great lecture series; secondly for providing us with a proof of a lemma that we were not able to find ourselves (although, in hindsight, we really should have been) and thirdly for a number of valuable comments and corrections on a preliminary version of these notes. Lastly, it is a great pleasure to thank the organizers, Damien Gaboriau, Sorin Popa and Stefaan Vaes, of the focused semester on \emph{von Neumann algebras and ergodic theory of group actions} where all this took place.\\

\vspace{0.2cm}
\noindent{\emph{Standing assumptions and notation:}}
Throughout the text, all generic von Neumann algebras are assumed to have separable predual and all discrete groups are assumed to be countable. Furthermore, all group actions on measure spaces are assumed to be essentially free, ergodic and measure preserving unless specified otherwise. 
For a finite von Neumann algebra $M$ with a fixed normal, faithful tracial state $\tau$ we denote by $L^2(M)$ the GNS-space arising from $\tau$ and consider $M$ as included in $B(L^2(M))$. For a vector $\xi\in L^2(M)$ and an operator $a\in M$ we will denote by $\xi a$ the vector $Ja^*J\xi$, where $J$ is the modular conjugation arising from the trace $\tau$. Lastly, we will abbreviate the strong operator topology SOT and denote the operator norm on $M$ and the 2-norm on $L^2(M)$ by $\|\cdot\|_\infty$ and $\|\cdot\|_2$, respectively.

\section{Introduction and statement of the results}
Consider a countable discrete group $\Gamma$ acting essentially freely, ergodically and probability measure preserving (p.m.p.) on a standard 
Borel space $(X,\mu)$. This induces a trace preserving action of $\Gamma$ on $L^\infty(X)$, given by $(\gamma\cdot f)(x):=f(\gamma^{-1}x)$,  and the crossed product von Neumann algebra $M:=L^\infty(X)\rtimes \Gamma$ is a $\I\I_1$-factor. 
One can consider several equivalence relations between such actions and we will here be concerned with the following three:
\begin{defi}
Two actions $\Gamma \curvearrowright (X,\mu)$ and $\Lambda \curvearrowright (Y,\nu)$ are called
\begin{itemize}
\item \emph{Conjugate} if there exists a group isomorphism $\delta\colon \Gamma \to \Lambda$ and a measure space isomorphism $\theta\colon X\to Y$ such that $\theta(\gamma x)=\delta(\gamma)\theta (x)$ for all $\gamma\in \Gamma$ and all\footnote{here, and throughout the text, ``all $x\in X$'' really means ``for $\mu$-almost all $x\in X$''.} $x\in X$.
\item \emph{Orbit equivalent} if there exists a measure space isomorphism $\theta\colon X\to Y$ such that $\theta(\Gamma x)=\Lambda\theta(x)$ for all $x\in X$.
\item \emph{$W^*$-equivalent} if $L^\infty(X)\rtimes \Gamma$ is isomorphic to $L^\infty(Y)\rtimes \Lambda$.
\end{itemize}
\end{defi}
Clearly conjugate actions are orbit equivalent and due to the following theorem orbit equivalence implies $W^*$-equivalence.
\begin{thm}[Singer, \cite{singer-1955}]
Two actions $\Gamma \curvearrowright (X,\mu)$ and $\Lambda \curvearrowright (Y,\nu)$ are orbit equivalent if and only if there exists a  $*$-isomorphism $\varphi\colon L^\infty(X)\rtimes \Gamma \to L^\infty(Y)\rtimes \Lambda$ such that $\varphi(L^\infty(X))=L^\infty(Y)$.
\end{thm}
Recall that the subalgebra $A:=L^\infty(X)\subset L^\infty(X)\rtimes\Gamma=:M$ is an example of a  \emph{Cartan subalgebra}; i.e.~it is a maximal abelian subalgebra whose normalizer 
\[
\N_M(A):=\{u\in \U(M)\mid uAu^*=A\}
\]
generates $M$ as a von Neumann algebra. The main problem in the field of deformation/rigidity and $W^*$-super rigidity is to analyze under which circumstances (one or both of) the implications 
\vspace{0.2cm}
\begin{center}
Conjugacy $\Rightarrow$ Orbit equivalence $\Rightarrow$ $W^*$-equivalence
\end{center} 
\vspace{0.2cm}
can be reversed, and the action $\GammaÊ\curvearrowright (X,\mu)$ itself is said to be \emph{$W^*$-super rigid} if any other essentially free, ergodic p.m.p.~action of any other group which is $W^*$-equivalent to it is actually conjugate to it. This, of course, is a rather rare phenomenon, which is illustrated well by Connes' results \cite{con76} showing that any two essentially free, ergodic p.m.p.~actions of any two amenable groups are $W^*$-equivalent.  Thus, actions of amenable groups are always highly non-$W^*$-super rigid.
However, recent advances within the field of deformation/rigidity have provided a number of interesting examples of $W^*$-super rigid actions. For example, it is proved by Popa and Vaes \cite{popa-vaes-grp-measure-space-decomp} that the Bernoulli action of a certain class of amalgamated free product groups is $W^*$-super rigid and, even stronger, Houdayer, Popa and Vaes prove in \cite{houdayer-popa-vaes} that \emph{any} essentially free, ergodic p.m.p.~action of $SL(3,\ZZ)\ast_T SL(3,\ZZ)$ (the amalgamation being over the subgroup of matrices $(x_{ij})$ with $x_{31}=x_{32}=0$) is  $W^*$-super rigid. Another very striking result in this direction is due to Ioana \cite{ioana-T-bernoulli}, who proved  that the Bernoulli action of an arbitrary property $\T$ group is $W^*$-super rigid. \\

Let us now return to our generic essentially free, ergodic p.m.p~action $\Gamma\curvearrowright (X,\mu)$. The general strategy for proving $W^*$-super rigidity for such an action is the following:
\begin{itemize}
\item Firstly, prove that the crossed product von Neumann algebra $M:=L^\infty(X)\rtimes \Gamma$ has only one \emph{group-measure space Cartan}  up to unitary conjugacy; i.e.~that for any other group-measure space decomposition $M=L^\infty(Y)\rtimes \Lambda$ there exists a unitary $u\in M$ such that $uL^\infty(X)u^*=L^\infty(Y)$. 
\item Then apply Singer's result (see above) to pass from $W^*$-equivalence with another action to orbit equivalence with it.
\item Lastly, appeal to results from the theory of orbit equivalence super rigidity to pass from orbit equivalence to actual conjugacy.

\end{itemize}
In the sequel we will only be concerned with the first part of this strategy and the main goal is to prove the following.

\begin{thm}[Chifan-Peterson, \cite{chifan-peterson}]\label{chifan-peterson-thm}
Assume that $\Gamma$ is a discrete countable group acting ergodically, essentially freely and measure preserving on a standard probability space $(X,\mu)$.  Assume furthermore that $\Lambda$ is another group acting (essentially  freely, ergodically and p.m.p.) on another probability space $(Y,\nu)$  such that $M=L^\infty(Y)\rtimes \Lambda$. If
\begin{itemize}
\item there exists an infinite subgroup $\Gamma_0< \Gamma$ with property $\T$ and
\item a mixing representation $\pi\colon \Gamma\to O(H)$ on a real Hilbert space $H$ with an unbounded 1-cocycle $b\colon \Gamma \to H$,
\end{itemize}
then there exists a unitary $u\in M$ such that $uL^{\infty}(X)u^*=L^\infty(Y)$. 
\end{thm}
The reader not familiar with Kazhdan's notion of property $\T$ is referred to the monograph \cite{BHV} for a detailed treatment of the subject. Recall that a representation $\pi\colon \Gamma \to  O(H)$ is called \emph{mixing} if $\lim_{g\to \infty}\ip{\pi(g)\xi}{\eta}=0$ for all $\xi,\eta \in H$.\footnote{The limit $\lim_{g\to \infty}\ip{\pi(g)\xi}{\eta}=0$  means that $\forall \varps>0 \ \exists F\subset \Gamma$ finite $\forall g\in \Gamma\setminus F$: $|\ip{\pi(g)\xi}{\eta}|<\varps$.} We also recall that a map $b\colon \Gamma \to H$  is called a \emph{1-cocycle} if it satisfies $b(gh)=\pi(g)b(h) +b(g)$ for all $g,h\in \Gamma$. A somewhat trivial source of cocycles is the class of so-called
 inner cocycles; these are the ones  given by $b(g)=\pi(g)\xi-\xi$ for a vector $\xi \in H$. Note that an unbounded cocycle is the same as a non-inner cocycle \cite[Proposition 2.2.9]{BHV}  and hence the Delorme-Guichardet theorem \cite[Theorem 2.12.4]{BHV}) implies that the property $\T$ subgroup  $\Gamma_0$ can never possess such a cocycle. It is therefore imperative that the subgroup $\Gamma_0$ is a proper subgroup.
%Put even more heuristically: the assumption that $\Gamma_0$ is infinite and has property $\T$ provides a substantial amount of %rigidity whereas the assumption that $\Gamma$ allows for an unbounded cocycle can be thought of as ensuring that there is some %still ''softness'' present outside of $\Gamma_0$.\\

\begin{ex}
Note that any free product $\Gamma=\Gamma_1\ast \Gamma_2$ with $\Gamma_1$ an infinite property $\T$ group and $\Gamma_2$ non-trivial satisfies the assumptions in TheoremÊ \ref{chifan-peterson-thm}. The existence of an infinite property $\T$ subgroup is clear and one can define an unbounded cocycle $b\colon \Gamma \to \ell^2(\Gamma,\RR)$ by setting
\[
b(g) =
\left\{
	\begin{array}{ll}
		\delta_g-\delta_e  & \mbox{if } g \in \Gamma_1 \\
		0 & \mbox{if }g\in \Gamma_2
	\end{array}
\right.
\]
and extending by the cocycle relation. To see that $b$ is unbounded, take two  nontrivial elements $g_1\in \Gamma_1$ and $g_2\in \Gamma_2$ and prove by induction that $\|b((g_1g_2)^n)\|_2=\sqrt{2n}$. This idea can be extended to cover certain amalgamated free products as well; see \cite{vaes-cohomology-and-unique-cartan} for the state of the art in this direction.
\end{ex}
Another important result in the same direction is the following theorem due to Popa and Vaes which is the precursor of the Chifan-Peterson theorem stated above. Before stating the result we  define the class of groups it deals with: Denote by $\G$  the class of discrete countable groups $\Gamma$ which can be written as an amalgamated free product $\Gamma_1 \ast_\Sigma \Gamma_2$ such that the following holds
\begin{itemize}
\item[(a)] $\Gamma_1$ contains a non-amenable subgroup $\Lambda$ with the relative property $\T$ or two non-amenable commuting subgroups.
\item[(b)] $\Sigma$ is amenable and  $\Gamma_2$ is strictly bigger than $\Sigma$.
\item[(c)] There exist $g_1,\dots, g_n \in \Gamma$ such that $\cap_{i=1}^n g_i\Sigma g_i^{-1}$ is finite.
\end{itemize}
The statement now is as follows.
\begin{thm}[Popa-Vaes \cite{popa-vaes-grp-measure-space-decomp}]\label{popa-vaes-thm}
Let $\Gamma$ be a group in $\G$ and assume that $\Gamma$ acts essentially freely, ergodically and measure preserving on a standard probability space $(X,\mu)$. Define $M:=L^\infty(X)\rtimes \Gamma$. If $\Lambda$ is any other group acting essentially freely, ergodically and measure preserving on another standard probability space $(Y,\nu)$ such that $M=L^\infty(Y)\rtimes \Lambda$ then there exists a unitary $u\in M$ such that $L^\infty(Y)=uL^\infty(X)u^*$.
\end{thm}
As mentioned already, the main goal is to prove Theorem \ref{chifan-peterson-thm}. The method of proof differs from the original paper \cite{chifan-peterson};  we are going to follow the course of the proof given in \cite{vaes-cohomology-and-unique-cartan}, although we are able circumvent some technicalities since we restrict attention to a smaller class of groups.  Note also that  the ``fully-fledged''  version of the result presented in  \cite{vaes-cohomology-and-unique-cartan} can  be used to derive the original result of Popa and Vaes stated above and therefore covers all known results regarding uniqueness of group-measure space Cartan algebras.

\section{Tools and strategy of proof}
In this section we gather the tools we need to import from the general theory of deformation/rigidity. The main results are stated without proofs, but all missing proofs can be found, for instance,  in \cite[Appendix F]{brown-ozawa}.  Actually,  for the reader not familiar with  the theory of bimodules (correspondences) over von Neumann algebras the introduction given in \cite[Appendix F]{brown-ozawa} is highly recommendable; basically all the  results needed are explained and proved in about 10 pages. Our main tool is Popa's \emph{intertwining by bi-modules technique} which is summarized in the following theorem.

\begin{thm}[Popa, \cite{popa-malleable-actions-I}]\label{fundamental-thm}
For a finite von Neumann algebra $(M,\tau)$ with von Neumann subalgebras $A,B\subseteq M$ the following are equivalent
\begin{itemize}
\item[(i)] There exists a Hilbert $A-B$-bimodule $H\subseteq L^2(M)$ such that $\dim_B H< \infty$.
\item[(ii)] There exist non-zero projections $p\in A$ and $q\in B$, a normal, unital $*$-homo\-morphism $\theta\colon pAp \to qBq $  and a non-zero partial isometry $v\in M$ such that $v^*v\in \theta(pAp)' \cap qMq$, $vv^*\in (pAp)'\cap pMp$ and
\[
xv=v\theta(x) \ \text{ for all }\  x\in pAp.
\]
\item[(iii)] There is no sequence  $v_n\in \U(A) $ such that $\lim _n\|E_B(xv_n y)\|_2=0 $ for all $x,y\in M$.

\end{itemize}
\end{thm}
\begin{rem}
If one (and hence all) of the conditions in the above theorem is fulfilled we say that \emph{(a corner of) A embeds into (a corner of) B inside M} or that \emph{A can be virtually conjugated into B inside M};  in symbols $A\prec_M B$. We will refer to condition (iii) as \emph{Popa's criterion} in what follows. Since the homomorphism $\theta$ in part (ii) is normal its kernel is a weakly closed 2-sided ideal in $pAp$. Hence, by cutting $p$ with a central projection in $pAp$ (the complement of the one defining the kernel of $\theta$) we can always arrange for $\theta$ to be injective. In most applications we will therefore implicitly assume that  $\theta$ is an honest embedding.

\end{rem}

Since we are primarily interested in the case where $A$ and $B$ are Cartan subalgebras we now analyze this situation in more detail.
Firstly, in the case of Cartan subalgebras virtual conjugacy implies actual conjugacy as the following result shows.
\begin{cor}[Popa, \cite{popa-malleable-actions-I}]\label{cartan-virtual}
If $A,B\subseteq M$ are two Cartan subalgebras and $A\prec_M B$ then there exists $u\in \U(M)$ such that $B=uAu^*$.
\end{cor}

In our situation, $M$ is always arising from a crossed product construction and in this situation the condition (iii) in Theorem \ref{fundamental-thm} (or rather its negation) allows the following more intuitive interpretation.

 \begin{prop}\label{fourier-prop}
Assume that  $B$ is a von Neumann algebra and that $M=B\rtimes \Gamma$ for a  group $\Gamma$ and a trace preserving action $\sigma\colon \Gamma \to \Aut(B)$, and  denote by $(u_g)_{g \in \Gamma}$ the natural unitaries implementing the action. Let $v_n\in M$ be unitaries and expand them in $L^2(M)$ as $v_n=\sum_{g\in \Gamma} (v_n)_g u_g$. Then $\| E_B(xv_ny)\|_2\to 0$ for all $x,y\in M$ if and only if the Fourier coefficients $(v_n)_g$ go to $0$ in 2-norm for every $g \in \Gamma$. 
\end{prop}
\begin{proof}
Since $(v_n)_g:=E_B(v_n u_g^*)$ one implication is obvious. To prove the opposite implication we first consider $x,y$ of the form $x=x_{g_0}u_{g_0}$ and $y=y_{h_0}u_{h_0}$  and compute:
\begin{align*}
\|E_B(x v_n y)\|_2 &=  \|E_B\sum_g x (v_n)_g u_g y \|_2\\
&= \|  E_B\sum_g x_{g_0}u_{g_0}(v_n)_gu_g y_{h_0}u_{h_0}  \|_2 \\
&=\| E_B\sum_gx_{g_0} \sigma_{g_0}((v_n)_g)u_{g_0g}y_{h_0}u_{h_0}    \|_2\\
&= \| E_B \sum_{g}x_{g_0}\sigma_{g_0}((v_n)_g)\sigma_{g_0g}(y_{h_0})u_{g_0gh_0}   \|_2\\
&=\|x_{g_0}\sigma_{g_0}((v_n)_{g_0^{-1}h_{0}^{-1}})\sigma_{h_0^{-1}}(y_{h_0})  \|_2\\
&\leq \|x_{g_0}\|_\infty \|y_{h_0}\|_\infty \|\sigma_{g_0}((v_n)_{g_0^{-1} h_0^{-1}})\|_2\\
&= \|x_{g_0}\|_\infty \|y_{h_0}\|_\infty \|(v_n)_{g_0^{-1} h_0^{-1}}\|_2,\\
\end{align*}
where the last equality follows since the action is trace preserving and each of the automorphisms $\sigma_g$ therefore extends to a unitary on $L^2(M)$. By assumption the last expression converges to zero and extending additively we get the desired convergence for $x,y\in B\rtimes_{\alg}\Gamma$. Using Kaplansky's density theorem one then extends the convergence to $x\in(B\rtimes \Gamma)_1$ and $y\in (B\rtimes_{\alg}Ê\Gamma)_1$ and, in turn,  to arbitrary  $x,y \in B\rtimes \Gamma$.
\end{proof}

%Lastly the following result will be useful.
%\begin{prop}
%If $A\subseteq M$ is Cartan and $p\in A$ a projection then $pAp\subseteq pMp$ is Cartan
%\end{prop}
%\note{do we use it?}
\subsection*{Strategy of proof}
The proof of Theorem \ref{chifan-peterson-thm} is quite involved and we will proceed in a series of steps. The first step consists in building another von Neumann algebra $\tilde{M}$ containing $M$ and a \emph{malleable deformation} (in the sense of Popa) of this algebra; i.e.~a one-parameter family of automorphisms $\alpha_t$ on $\tilde{M}$ converging to the identity in point-2-norm. This is done in the following section. The next main part of the program is to use this deformation  to show that the rigidity arising from the property $\T$ subgroup $\Gamma_0$ can be transferred into the the other picture of $M$ as $B\rtimes \Lambda$ by means of a certain sequence of elements in $\Lambda$; this is done in Section \ref{transfer-section}. In Section \ref{implementation-section} we  show how the transferred rigidity can be used to prove  that the deformation $\alpha_t$ is implemented by a partial isometry on $B$, and in section \ref{conjugacy-section} we then show how to pass from such an implementation to conjugacy of $A$ and $B$.

\section{Deformations arising from cocycles}\label{deform-from-cocycles}
In this section we show, following the approach of Thomas Sinclair \cite{sinclair}, how to construct a von Neumann algebra with a so-called \emph{malleable deformation} from a $1$-cocycle on a group. This construction will be essential in the proof of our main theorem.\\

Consider again our standard setup:  $\Gamma$ is the countable group from Theorem \ref{chifan-peterson-thm}, $\pi\colon \Gamma \to H$ a mixing orthogonal representation and $b\colon \Gamma \to H$ an unbounded cocycle. Moreover, $\Gamma$ acts on a standard probability space $(X,\mu)$ and we put $A:=L^\infty(X) $, $M:=A\rtimes \Gamma$ and denote by $(u_g)_{g\in \Gamma}$ the natural unitaries implementing the action.  Fix and orthonormal basis $\{e_n\}_{n\in \NN}$ in $H$ and put $H_0:=\spann_\RR\{e_n\mid n\in \NN\}$. Consider now the measure space
\[
(Z,\nu):= \prod_{n\in \NN} \left(\RR,  \tfrac{1}{\sqrt{2\pi}}e^{\sfrac{-x^2}{2}}dx \right)
\]
and define a map $\omega\colon H_0\to \U(L^\infty(Z))$ by 
\[
\omega \left(\sum_n c_ne_n \right): (z_n)_n \longmapsto \exp( \sqrt{2}i\sum_nc_n z_n ).
\]
As the following two lemmas show, the map $\omega$ ``exponentiates'' the additive group structure on $H$ into the multiplicative structure on the unitary group $\U(L^\infty(Z))$.

\begin{lem}\label{omega-formler}
We have $\omega(\xi +\eta)=\omega(\xi)\omega(\eta)$, $\omega(-\xi)=\omega(\xi)^*$ and $\tau(\omega(\xi))=e^{-\|\xi\|^2}$ for all $\xi,\eta \in H_0$.
\end{lem}
\begin{proof}
The two first claims follow from straightforward calculations and the last is seen as follows: Write $\xi=\sum_n c_n e_n $ and recall that $\frac{1}{\sqrt{2\pi}}\int_{\RR} e^{tx}e^{-\sfrac{x^2}{2}}dx=e^{\sfrac{t^2}{2}}$ for all $t\in \CC$. Then 
\begin{align*}
\tau(\omega(\xi)) &= \int_{(z_n)_n\in Z} \omega(\xi)((z_n)_n) d\nu((z_n)_n) \\
&= \int_{(z_n)_n\in Z} \prod_{n\in \NN} e^{\sqrt{2} i c_n z_n }  d\left ( \otimes_{n\in \NN} \tfrac{1}{\sqrt{2\pi}}e^{\sfrac{-z_n^2} {2}}dz_n \right )   \\
&= \prod_{n\in \NN} \tfrac{1}{\sqrt{2\pi}}\int_{\RR}e^{\sqrt{2}ic_nz_n } e^{\sfrac{-z_n^2}{2}}dz_n\\
&=\prod_{n\in \NN} e^{\frac{(\sqrt{2}i c_n)^2}{2}}=e^{-\sum_n c_n^2} =e^{-\|\xi\|^2}.
\end{align*}
\end{proof}
\begin{lem}
The map $\omega\colon H_0 \to \U(L^\infty(Z))$ is $\|\cdot\|_H$--SOT continuous and therefore extends to a map $\omega\colon H\to \U(L^\infty(Z))$ with the properties described in Lemma \ref{omega-formler}.
\end{lem}
\begin{proof} Let $\xi,\eta \in H_0$. Using Lemma \ref{omega-formler} we get
\begin{align*}
\|\omega(\xi)-\omega(\eta)\|_2^2 &= \tau((\omega(\xi)-\omega(\eta))^* (\omega(\xi)-\omega(\eta)))\\
&= \tau(1+1 - \omega(-\xi +\eta) -\omega(-\eta +\xi))\\
&= 2(1-e^{-\|\xi-\eta\|^2}) 
\end{align*}
Hence $\omega$ is $\|\cdot\|_H$ -- $\|\cdot\|_{2,\tau}$ continuous. Since the SOT and the 2-norm topology coincides on $\U(L^\infty(Z))$ and since $\U(L^\infty(Z))$ is sequentially SOT-complete (this can be proved easily by hand) it follows that $\omega$ extends.
\end{proof}
Define $D_0:=\spann_\CC \{\omega(\xi)\mid \xi \in H\}$ and note that $D_0$ is a $*$-subalgebra in $L^\infty(Z)$. Denote by $D$ the enveloping von Neumann algebra $D_0''\subseteq L^\infty (Z)$ and endow $D$ with the  $\Gamma$-action given by
\[
\gamma\cdot \omega(\xi):=\omega(\pi(\gamma)\xi)
\]
Since
\[
\tau(\gamma\omega(\xi))=\tau(\omega(\pi(\gamma)x))=e^{-\|\pi(\gamma)\xi\|^2}=e^{-\|\xi\|^2}=\tau(\omega(\xi))
\]
this action is trace preserving on $D_0$ and thus on all of $D$ by normality. The action $\Gamma \curvearrowright D$ is called the \emph{Gaussian action}.  Consider now the tensor product $D\htens A$ endowed with the diagonal $\Gamma$-action $\sigma$ and define $\tilde{M}:= (D\htens A)\rtimes \Gamma$. We will consider $M=A\rtimes \Gamma$ as a subalgebra of $\tilde{M}$ via its image $(1\tens A)\rtimes \Gamma$ therein. Note that the natural unitaries  $\tilde{u}_g$ implementing the action $\Gamma \curvearrowright D\htens A$ have the property that
\[
\tilde{u}_g (1\tens a)\tilde{u}_g^*= 1\tens u_gau_g^*,
\]
and we will therefore consider them as natural extension of the $u_g$'s and just denote $\tilde{u}_g$ by $u_g$ in the following.
Next we construct a deformation of $\tilde{M}$ as follows: for $t\in \RR$ define $\alpha_t\colon \tilde{M} \to \tilde{M}$ by
\[
\alpha_t|_{D\htens A}=\id \quad \text{ and } \quad \alpha_t(u_g)=(\omega(tb(g)) \tens 1)u_g.
\]
The word ``deformation'' here refers to Popa's notion of \emph{malleable deformation}, the precise meaning of which is contained in
the following result.

\begin{prop}
For every $t\in \RR$ we have $\alpha_t\in \Aut(\tilde{M})$ and $\lim_{t\to 0}\|\alpha_t(x)-x\|_2=0$ for every $x\in \tilde{M}$. 
\end{prop}
\begin{proof}
To prove that $\alpha_t$ is a $*$-homomorphism it suffices to show that
\[
\alpha_t(u_{gh})=\alpha_t(u_g)\alpha_t(u_h) \ \text{ and } \ \alpha_t(u_{g^{-1}})=\alpha_t(u_g)^*.
\]
Using the cocycle identity, the multiplicativity follows from a straight forward calculation:
\begin{align*}
\alpha_t(u_g)\alpha_t(u_h) &=  (\omega(tb(g))\tens 1) u_g   (\omega(tb(h))\tens 1) u_h   \\
&= (\omega(tb(g))\tens 1  )u_g(\omega(tb(h))\tens 1)u_g^* u_{gh}^{\phantom{*}}\\
&=(\omega(tb(g))\tens 1) (\omega (t\pi(g)b(h)  ) \tens 1)u_{gh}\\
&=(\omega(t (b(g) +\pi(g)b(h))) \tens 1 )u_{gh}\\
&=(\omega(tb(gh))\tens 1)u_{gh}\\
&=\alpha_t(u_{gh}).
\end{align*}
Noting that $0=b(e) = b(g^{-1}g)=\pi(g^{-1})b(g) + b(g^{-1})$ we get
\begin{align*}
\alpha_t(u_{g^{-1}}) &=(\omega(tb(g^{-1}))\tens 1)u_{g^{-1}}\\
&= (\omega(-t \pi(g^{-1})b(g) )\tens 1)u_{g}^*\\
&=\sigma_{g^{-1}}( \omega(-tb(g)\tens 1) )u_g^*\\
&=u_{g}^*(\omega(tb(g))\tens 1)u_{g}^{\phantom{*}}u_g^*\\
&=\alpha_t(u_g)^*.
\end{align*}
Hence $\alpha_t$ is a $*$-endomorphism of $\tilde{M}$. Since $\alpha_{t}\alpha_{s}=\alpha_{t+s}$ and $\alpha_0=\id$ it follows that $\alpha_t$ is a 1-parameter group of $*$-automorphisms.\\

To prove the convergence statement it suffices (by Kaplansky's density theorem) to treat elements in $(D\htens A)\rtimes_{\alg}\Gamma$ and since $\alpha_t {|}_{D\htens A}=\id$ it therefore suffices to prove it for elements of the form $x=u_g$ for $g\in \Gamma$. For such an element we
 simply calculate the 2-norm:
\begin{align*}
\|\alpha_t(u_g)-u_g\|_2^2 &=\tau_{\tilde{M}}((\alpha_t(u_g) - u_g)^*(\alpha_t(u_g) -u_g))\\
&=2 -\tau_{\tilde{M}}(\alpha_t(u_g)^* u_g)-\tau(u_g^*\alpha_t(u_g))\\
&=2-\tau_{\tilde{M}}((\omega(tb(g^{-1})) \tens 1  )u_{g}^*u_g^{\phantom{*}}  ) -\tau_{\tilde{M}}(u_g^*(\omega(tb(g))\tens 1 )u_g  )\\
&= 2 - \tau_{D\htens A}(\omega(tb(g^{-1}))\tens 1)- \tau_{D\htens A} (\omega(tb(g))\tens 1)\\
&= 2-e^{-t^2\|b(g^{-1})\|^2}-e^{-t^2\|b(g)\|^2}\underset{t\to 0}{\To} 0.
\end{align*}

\end{proof}

Actually, the convergence of $\alpha_t$  is monotone on $M$ as the following result shows.
\begin{prop}\label{monotone-prop}
If $r\geq s\geq 0$ then $\|\alpha_r(x)-x\|_2 \geq \|\alpha_s(x)-x\|_2$ for all $x\in M$.
\end{prop}
\begin{proof}
Expand $x$ as $\sum_{g\in \Gamma}(1\tens x_g)u_g \in M\subset \tilde{M}$ and compute:
\begin{align}
\|\alpha_r(x)-x  \|_2^2 &=\|\sum_{g\in \Gamma} (\omega(rb(g)) -1 )\tens x_g)u_g\|_2^2\notag\\
&=\sum_{g\in \Gamma} \|\omega (rb(g))-1  \|_2^2 \|x_g\|_2^2\notag\\
&=\sum_{g\in \Gamma} 2(1-e^{-r\|b(g)\|^2}) \|x_g\|_2^2 \tag{using Lemma \ref{omega-formler}}
\end{align}
If $r\geq s\geq 0$ then $1-e^{-r\|b(g)\|^2}\geq 1-e^{-s\|b(g)\|^2}$ and hence $\|\alpha_r(x)-x  \|_2^2\geq \|\alpha_s(x)-x  \|_2^2$ as desired.

\end{proof}

Next we need to analyze the deformation $\alpha_t$ in more detail. This is done in sequence  lemmas, which may a first glance seem technical and not that easy to access.
However, despite their somewhat technical nature these lemmas really reveal a lot of information about the behavior of the deformation  --- so hang in there! Heuristically,  the first lemma shows that if an element $x\in M$ is moved close to $M$ by $\alpha_t$ then actually $\alpha_t(x)$ has to be close to $x$.
\begin{lem}\label{funky-lem}
For every $t\in \RR$ and every $x\in M$ we have $\tau_{\tilde{M}}(\alpha_t(x)x^* )\geq \| E_M(\alpha_t(x))\|_2^2$.
\end{lem}
\begin{proof}
Take $x=\sum_g(1\tens a_g)u_g\in M=(1\tens A)\rtimes \Gamma$ and compute the left hand side:
\begin{align*}
\tau_{\tilde{M}}(\alpha_t(x)x^*) &= \tau_{\tilde{M}} \left( \left(\sum_g (1\tens a_g)(\omega(tb(g))\tens 1)u_g\right)\left(\sum_ h u_h^*(1\tens a_h^*)\right)  \right ) \\
&= \tau_{\tilde{M}}Ê\left ( \sum_{g, h} (\omega(tb(g)) \tens a_g ) u_{gh^{-1}}(1\tens a_h^*)  \right )\\
&= \tau_{\tilde{M}}Ê\left ( \sum_{g, h} (\omega(tb(g)) \tens a_g ) \sigma_{gh^{-1}}(1\tens a_h^*)  u_{gh^{-1}}  \right )\\
&= \tau_{D\htens A}Ê \left ( E_{D\htens A}\left( \sum_{g, h} (\omega(tb(g)) \tens a_g ) \sigma_{gh^{-1}}(1\tens a_h^*)  u_{gh^{-1}}\right)  \right )\\
&=\tau_{D\htens A} \left(\sum_g \omega(tb(g)) \tens a_g^{\phantom{*}}a_g^*   \right) \\
&=\sum_g e^{-\| tb(g)\|^2}\|a_g\|_2^2
\end{align*}
Using that  $E_M$ is  simply given by $(d\tens a)u_g\mapsto \tau(d) au_g$, computing the right hand side we get 
\begin{align*}
\|E_M(\alpha_t(x)) \|_2^2 &= \left\| E_M\left( \sum_g(1\tens a_g)(\omega(tb(g))\tens 1)u_g \right) \right\|_2^2\\
&= \left\| \sum_g \tau(\omega(tb(g)))a_g u_g \right\|_2^2\\
&= \sum_{g}|\tau(\omega(tb(g)))|^2 \|a_g\|_2^2\\
&= \sum_g e^{-2\| tb(g)\|^2}\|a_g\|_2^2
\end{align*}
Hence the two expressions differ by a factor of two in the exponent and the inequality follows.
\end{proof}

As a consequence of Lemma \ref{funky-lem} we  now get the following result which shows if a subalgebra is not moved to close to the orthogonal complement of $M$ by the deformation $\alpha_t$, then on this subalgebra the deformation is ``implemented'' by a partial isometry. The precise statement is as follows:

\begin{cor}\label{implement-cor}
Let $Q\subseteq M$ be a von Neumann subalgebra and let $\delta>0$ and $t\in \RR$ be given. Assume that $\|E_M(\alpha_t(x))\|_2\geq \delta$ for all $x\in \U(Q)$. Then there exists a non-zero partial isometry $v\in \tilde{M}$ such that $\alpha_t(x)v=vx$ for all $x\in Q$. 
 \end{cor}
\begin{proof}
By Lemma \ref{funky-lem} we have $\tau(\alpha_t(x)x^*)\geq \delta$ for all $x\in \U(Q)$ and therefore $0$ is not in the set
\[
K:= \overline{\conv}^{\|\cdot\|_2}\{ \alpha_t(x)x^*\mid x\in \U(Q)\}. 
\]
(if it were, then a sequence in the non-closed convex hull would converge in 2-norm to zero and since the 2-norm topology and the SOT coincides on bounded sets, the sequence would go to zero in SOT; hence its trace would go to zero contradicting the estimate
$\tau(\alpha_t(x)x^*)\geq \delta>0$). We therefore get a unique non-zero vector $\eta \in K$ of minimal 2-norm \cite[Proposition 2.2.1]{KR1}.   Note that for $x\in \U(Q)$ the map\footnote{Recall that we write $Ju^*J\eta$ as $\eta u$ for notational convenience.} $\xi\mapsto \alpha_t(x)\xi x^*$ maps $K$ isometrically onto itself and by uniqueness of $\eta$ we therefore get
\[
\alpha_t(x)\eta x^* =\eta \quad \text{for all } x\in \U(Q)
\]
Now, passing to $2\times 2$-matrices a direct computation shows that 
\[
\left [ \begin{array}{cc} 0 & \eta^*\\ \eta  &0  \end{array}  \right ] \text{ commutes with }   \left [ \begin{array}{cc} x & 0\\ 0  &\alpha_t(x)   \end{array}  \right ] \text{ for every } x\in \U(Q).
\]
Hence $\left [ \begin{array}{cc} 0 & \eta^*\\ \eta  &0  \end{array}  \right ]$ is affiliated with the commutant of the von Neumann algebra 
\[
N:=\left\{\left [ \begin{array}{cc} x & 0\\ 0  &\alpha_t(x)   \end{array}  \right ] \ \Big{|} \   x\in \U(Q) \right\}''  
\]
and, in particular, the partial isometry arising from its polar decomposition is an element in $N'$. Polar decomposing $\eta$ as $\eta=v|\eta |$ it is not difficult to see that this partial isometry is nothing but $\left [ \begin{array}{cc} 0 & v^*\\ v  &0  \end{array}  \right ]$ and writing out the relation that it commutes with $N$ yields $\alpha_t(x)v=vx$ for every $x\in \U(Q)$ and thus for every $x\in Q$.
\end{proof}

\begin{lem}[{\cite{popa-malleable-actions-I}}]\label{lem-1}
Let $v_n\in M$ be a 2-norm bounded sequence and assume that the Fourier coefficients (relative to the decomposition $M=A\rtimes \Gamma$) of $v_n$ go to zero in 2-norm. Then for all $x\in \tilde{M}$ and $y\in \tilde{M}\ominus M$ we have $\|E_M(xv_ny)\|_2\To 0$.
\end{lem}
If the sequence in question happens to be a sequence of unitaries, then one can think of the contents of Lemma \ref{lem-1} as follows: If the support of the Fourier decomposition of $v_n$ escapes every finite set in $\Gamma$ eventually, then  after squeezing $v_n$ by an element from $\tilde{M}$ and an element from $\tilde{M}\ominus M$ the resulting sequence is approaching $\tilde{M}\ominus M$. As the proof shows, Lemma \ref{lem-1} depends heavily on the representation $\pi\colon\Gamma \to O(H)$ being mixing.

\begin{proof}
Note that 
\[
\tilde{M}_0:= \spann_\CC\{(\omega(\xi) \tens a) u_g\mid \xi\in H, a\in A, g\in \Gamma  \}
\]
forms a dense $*$-subalgebra in $\tilde{M}$ with (see e.g.~Lemma \ref{omega-formler})
\[
(\id - E_M)(\tilde{M}_0)=\spann_{\CC}\{ ((\omega(\xi)-e^{-\|\xi\|^2}1 )\tens a)u_g \mid \xi\in H, a\in A,g \in \Gamma  \}.
\]
By Kaplansky's density theorem, it therefore suffices to consider $x$ and $y$ of the form $x= (\omega(\xi) \tens a) u_g$ and $y= ((\omega(\eta)-e^{-\|\eta\|^2}1 )\tens b)u_h$. We first re-write things slightly:
\begin{align*}
x &= (\omega(\xi) \tens a) u_g = (1\tens a) u_g \sigma_{g^{-1}}(\omega(\xi)\tens 1)= \underbrace{(1\tens a)u_g}_{\in M} (\omega(\underbrace{\pi_{g^{-1}}(\xi))}_{=:\xi'} \tens 1);\\
y&= ((\omega(\eta)-e^{-\|\eta\|^2}1 )\tens b)u_h =((\omega(\eta)-e^{-\|\eta\|^2}1 )\tens 1)\underbrace{(1\tens b)u_h}_{\in M}.
\end{align*}
Thus
\[
E_M(xv_ny) = (1\tens a)u_g E_M\left(\underbrace{(\omega(\xi')\tens 1  )v_n((\omega(\eta)-e^{-\|\eta\|^2}1)\tens 1)}_{=:z_n}\right) (1\tens b)u_h
\]
It therefore suffices to prove that $\|E_M(z_n)\|_2\to 0$. Expand $v_n\in M\subseteq \tilde{M}$ as $ v_n=\sum_{g}(1\tens (v_n)_g)u_g$. Then $z_n=\sum_g (d_g\tens (v_n)_g)u_g$ where 
\[
d_g:=\omega(\xi')(\omega(\pi_g(\eta))- e^{-\|\eta\|^2}1)=\omega(\xi'+ \pi_g(\eta))-e^{-\|\eta\|^2}\omega(\xi').
\]
 Thus
 \[
 \tau_D(d_g)=e^{-\|\xi'+\pi_g(\eta)\|^2 }- e^{-(\|\eta\|^2+\|\xi'\|^2)},
 \]
and since 
\[
\|\xi' +\pi_g(\eta)  \|_2^2=\|\xi'\|^2 +\|\eta\|^2 +\underbrace{2\Re\ip{\pi_g\xi'}{\eta}}_{\to 0 \text{ as } g\to \infty}
\]
(the convergence holds since $\pi$ is assumed to be mixing) we have $\tau_D(d_g)\to 0$ as $g\to \infty$. Since $E_M(z_n)=\sum_{g}\tau_D(d_g)(v_n)_gu_g$ we can, for given $\varps>0$, find a finite set $F\subseteq \Gamma$ (independent of $n$) such that for all $g\notin F$ we have $|\tau(d_g)|<\varps$. Thus
\begin{align*}
\|E_M(z_n)\|_2^2 &= \sum_g |\tau(d_g)|^2 \|(v_n)_g\|_2^2\\
&\leq \sum_{g\in F} |\tau(d_g)|^2 \|(v_n)_g\|_2^2 + \sum_{g\in \Gamma\setminus F}\varps^2 \|(v_n)_g\|_2^2\\
&\leq \sum_{g\in F} |\tau(d_g)|\|(v_n)_g\|_2^2 + \varps^2 \sup_n\|v_n\|_2^2. 
\end{align*}
Since $F$ is finite and independent of $n$ and we assume that $\lim_n\|(v_n)_g\|_2=0 $ for all $g\in \Gamma$ this implies that $\|E_M(z_n)\|_2\to 0$ as desired.

\end{proof}
\begin{defi}
For $F\subset \Gamma$ we denote by $P_F$ the orthogonal projection onto the space $\overline{\spann}^{\|\cdot\|_2}\{Au_g\mid g\in F\}\subseteq L^2(M)\subseteq L^2(\tilde{M}) $.
\end{defi}
Note that the Fourier coefficients of a sequence $v_n\in M$ go to zero in 2-norm iff $\|P_F(v_n)\|_2\to 0$ for all $F\subset \Gamma$ finite.
We will also need the following  ``uniform version'' of Lemma \ref{lem-1}.
\begin{cor}\label{cor-3}
Let $v_n\in \U(M)$ with Fourier coefficients going to zero in 2-norm and let $F\subset \Gamma$ be finite. Then for all $y\in \tilde{M}\ominus M$ we have 
\[
\sup_{x\in (M)_1}\|E_M(\alpha_t(P_F(x))v_n y)\|_2\to 0.
\]
\end{cor}
\begin{proof}
Let $x\in (M)_1$ be given and expand it as $x=\sum_{g\in \Gamma}(1\tens x_g)u_g\in M\subset\tilde{M}$. Then 
\[
\alpha_t(P_F(x))=\alpha_t\left(\sum_{g\in F}(1\tens x_g)u_g  \right)=\sum_{g\in F}(\omega(tb(g))\tens x_g )u_g= \sum_{g\in F}(1\tens x_g)\underbrace{(\omega(tb(g))\tens 1 )u_g}_{=:z_g}.
\]
(note that the $z_g$'s are independent of $x$; they only depend on $F$). Thus
\begin{align*}
\|E_M\left(\alpha_t(P_Fx)v_n y   \right)\|_2 &=\|\sum_{g\in F} E_M(  (1\tens x_g) z_g v_n y  )\|_2\\
&=\| \sum_{g\in F}  (1\tens x_g)E_M(z_g v_n y) \|_2\\
&\leq \sum_{g\in F} \|x_g\|_{\infty} \|E_M(z_g v_n y)\|_2\\
&\leq \sum_{g\in F}  \|E_M(z_g v_n y)\|_2,
\end{align*}
where the last inequality follows from $\|x_g\|_{\infty}=\|E_M (xu_g^*)\|_\infty\leq 1$  ($E_M$ is u.c.p.~and thus in particular contractive). The last sum in the above calculation is independent of $x$ and converges to zero by Lemma \ref{lem-1}, and hence the uniform convergence follows.

\end{proof}

The next result is a souped  up version of Lemma \ref{lem-1} (when $\delta=0$ the two statements coincide).

\begin{lem}\label{lem-2}
Let $v_n\in M$ be a sequence of unitaries and let $\delta\geq 0$. Assume that  for all finite  $F\subset \Gamma$ we have $\limsup_n\|P_F(v_n)\|_2\leq \delta$. Then for all $x\in (\tilde{M})_1$ and $y\in (\tilde{M}\ominus M)_1$ we have $\limsup_n \|E_M(xv_ny)\|_2\leq \delta$.
\end{lem}
\begin{proof}
Assume this is not the case; then there exists an $\omega_0>0$ and a subsequence $v_{l(n)}$ such that $\|E_M(x v_{l(n)}y)\|_2\geq \delta + \omega_0$ for all $n$. Since the limes superior of a subsequence is always bounded by that of the original sequence, the subsequence $v_{l(n)}$ still satisfies the assumptions in the lemma and we may therefore without loss of generality assume that $l(n)=n$. Let $F_k\subset \Gamma$ be an increasing sequence of finite subsets with union $\Gamma$ and let $\varps>0$ be given.
The assumption gives that $\limsup_n \|P_{F_k}(v_n)\|_2\leq \delta$ for each $k\in \NN$, so in particular we can find an index $n(k)$ such that $\|P_{F_k}(v_{n(k)})\|_2\leq \delta +\varps$. Consider now $w_{k}:= v_{n(k)}-P_{F_k}(v_{n(k)})$. Since any finite set $F\subset \Gamma$ is eventually contained in $F_k$'s we get that $\|P_F(w_{k})\|_2\to_k 0$ and hence the Fourier coefficients of the sequence $w_{k}$ go to zero. Applying Lemma \ref{lem-1}, we therefore get  $\|E_M(x w_{k}y)\|_2\to_k 0$; in particular there exists $k_0\in \NN$ such that  $\|E_M(xw_ky)\|_2\leq \varps $ when $k\geq k_0$.  For $k\geq k_0$ we  thus have
\begin{align}
\|E_M(xv_{n(k)} y)\|_2 &\leq \| E_M(x w_k y)  \|_2 + \| E_M(xP_{F_k}(v_{n(k)}) y) \|_2 \notag \\
&\leq \varps + \| xP_{F_k}(v_{n(k)}) y  \|_2 \tag{Schwarz inequality \cite[1.5.7]{brown-ozawa}} \\
&\leq \varps + \| P_{F_k}(v_{n(k)}) \|_2\tag{$x,y\in (\tilde{M})_1$}\\
&\leq \delta + 2\varps.\notag
\end{align}
On the other hand, since $v_{n(k)}$ is still an element in the original sequence we also have $\|E_M(xv_{n(k)} y)\|_2\geq \delta +\omega_0$ and since $\omega_0$ is fixed and $\varps$ was arbitrary this yields the desired contradiction.
\end{proof}
%brown-ozawa prop 1.5.7.
Dear reader, thank you for staying with us  through these technical pages. The reward is waiting in the following sections.

\section{Transferring the rigidity}\label{transfer-section}
In this section we consider again the setup from our main theorem: Thus $M=A\rtimes \Gamma$ where $\Gamma$ satisfies the condition in Theorem \ref{chifan-peterson-thm} and $(u_g)_{g\in \Gamma}$ denote the natural unitaries implementing the $\Gamma$-action. Assume furthermore that $M$ can also be written as $B\rtimes \Lambda$ for another group $\Lambda$ and another abelian von Neumann algebra $B$ and denote by $(v_h)_{h\in \Lambda}$ the natural unitaries implementing the $\Lambda$-action. We aim at proving the following \emph{transfer of rigidity} result.

\begin{thm}[Popa--Vaes, \cite{popa-vaes-grp-measure-space-decomp}]\label{transfer-thm}
For all $\delta>0$ there exists a sequence $h_n\in \Lambda$ and a $t>0$ such that $v_n:= v_{h_n}$ satisfy:
\begin{itemize}
\item[(1)] $ \|\alpha_t(v_n)-v_n\|_2\leq \delta  $ for all $n\in \NN$.
\item[(2)] The Fourier coefficients of $v_n$ (relative to the decomposition $M=A\rtimes \Gamma$) go to zero in 2-norm.
\end{itemize}
\end{thm}
The sequence $v_n$ can be thought of as ``the shadow'' of the property (T) subgroup $\Gamma_0\leq \Gamma$ in the decomposition arising from $\Lambda$. Define $\Delta\colon M\to M\htens M$ by
$
\Delta(bv_h)=bv_h\tens v_h
$
and put $Q:= L\Gamma_0\subseteq M$. The following lemma shows that the comultiplication\footnote{See any text on quantum groups or Hopf algebras for an explanation of this terminology.} $\Delta$ distributes the rigidity of $Q$ on both legs of the tensor products, so that we cannot virtually conjugate $\Delta(Q)$ into $M\tens A$ (the $A$-part is abelian and hence nothing rigid can be embedded into it).

\begin{lem}\label{non-embedding-lem}
$\Delta(Q)\nprec_{M\htens M} M\htens A$.
\end{lem}

\begin{proof}
Assume that $\Delta(Q)\prec_{M\htens M} M\htens A$; i.e.~that a corner $\Delta(qQq)$ actually embeds into a corner of $M\htens A$. Denote by $\theta\colon \Delta(qQq)\to M\htens A $ the embedding. The proof proceeds in two steps.\\

\noindent \textbf{Step 1:} We first show that the assumptions actually actually imply $\Delta(Q)\prec M\tens 1$. Since $A$ is abelian it is hyperfinite so we may write it as $A=(\cup_n A_n)''$ for an increasing family of finite dimensional algebras. Then the maps
\[
\xymatrix{\Phi_n\colon M\htens A  \ar[rr]^{\id \tens E_{A_n}}& &M\tens A_n \ar@{^{(}->}[r] & M\htens A}
\]
are easily seen to constitute a deformation of $M\tens A$; i.e.~the $\Phi_n$'s are u.c.p.~and converge pointwise in 2-norm to the identity. Since $\Gamma_0$ has property (T) the von Neumann algebra $\Delta(Q)$ has property (T) and is therefore rigid in any 
ambient von Neumann algebra\footnote{meaning that any deformation of the ambient von Neumann algebra  converges uniformly on $\Delta(Q)$.}. Since property (T) is easily  seen to pass to corners the same is true for $P:=\theta(\Delta(qQq))\subset M\htens A$ and therefore the deformation $\Phi_n$ converges uniformly on $P$; i.e.~
\[
\sup_{y\in(P)_1}\|\Phi_n(y)-y  \|_2\To 0.
\]
From this we now obtain that $P\prec_{M\htens A} M\tens A_n$ for some $n\in \NN$: because if this is not the case, then there exist unitaries $w_k\in P$ such that $\|E_{M\tens A_n}(xw_ky)\|_2\to_k 0$ for all $x,y\in M\htens A$, but $\Phi_n$ equals $E_{M\tens A_n}$ so by the uniform convergence we get
\[
1=\|w_k\|_2 \sim \|\Phi_n(w_k)\|_2=\|E_{M\tens A_n}(w_k)\|\sim 0 
\]  
which is a contradiction. Thus $P\prec_{M\htens A} M\tens A_n$ for some $n\in \NN$. Since $A_n$ is finite dimensional, $M\tens A_n\prec_{M\htens M} M\tens 1$ and the claim (i.e.~that $\Delta(Q)\prec_{M\htens M} M\tens 1$) follows.\\

\noindent \textbf{Step 2:} Next we show how the conclusion of Step 1 leads to a contradiction by showing that $\Delta(Q)\prec_{M\htens M} M\tens 1$ implies $Q\prec_M B$ (this is absurd since any non-trivial corner in $Q$ is diffuse with property (T) and $B$ is abelian and thus unable to contain any diffuse property (T) von Neumann algebra).
% a diffuse abelian von Neumann algebra is isomorphic to $L^\infty(S^1)=L(\ZZ)$ which cannot have (T) as $\ZZ$ does not have (T) 
To see this, we assume that $Q\nprec_M B$ and use Popa's criterion to get a sequence of unitaries $u_n\in Q$ with Fourier coefficients converging to zero with respect to the decomposition $M=B\rtimes \Lambda$. Now we show that $\|E_{M\tens 1}(x\Delta(u_n)y)\|_2\to 0$ for all $x,y\in M\htens M$ and this is a contradiction (again by Popa's criterion)  since we just showed that $\Delta(Q)\prec M\tens 1$ . To see that $\|E_{M\tens 1}(x\Delta(u_n)y)\|_2\to 0$ for all $x,y\in M\htens M$, it is enough to look at basic tensors $x=x_1\tens x_2$ and $y=y_1\tens y_2$. Furthermore we can assume that both $x_1,x_2$ and $y_1,y_2$ have finite Fourier expansions since $(B\rtimes_\alg \Lambda) \odot(B\rtimes_\alg \Lambda) $ is dense in $M\htens M$. We now get:
\begin{align*}
\|E_{M\tens 1} (x\Delta(u_n)y)\|_2 & =\left \| E_{M\tens 1} \left( (x_1\tens x_2)\left(\sum_{h\in \Lambda} (u_n)_hv_h\tens v_h  \right)(y_1\tens y_2)\right) \right \|_2 \\
&= \left \| \sum_{h\in \Lambda} x_1(u_n)_hv_hy_1 \tau_M(x_2v_hy_2)  \right\|_2 \\
&\leq  \sum_{h\in \Lambda} \|x_1\|_\infty \|y_1\|_\infty \|(u_n)_h\|_2  |\tau_M(x_2v_hy_2)| 
\end{align*}
Since both $x_2$ and $y_2$ are assumed to have finite Fourier expansions relative to the decomposition $M=B\rtimes \Lambda$, the factor $|\tau_M(x_2v_hy_2)|=|\tau_M((y_2x_2)v_h)=|\tau_B(E_B(y_2x_2 v_h))| $ is only non-zero  for a finite set of $h$'s in $\Lambda$;  this finite set only depends on $x_2$ and $y_2$ and is therefore, in particular, independent of $n$. Since the Fourier coefficients of the $u_n$'s are assumed to go to zero in 2-norm the whole sum therefore converges to zero.

\end{proof}
We are now ready to prove the transfer of rigidity theorem.
\begin{proof}[Proof of Theorem \ref{transfer-thm}]
 Let $\delta>0$ be given and consider the deformation $\id\tens \alpha_t$ of $M\htens \tilde{M}$. Since  $Q$ has property (T) the inclusion  $\Delta(Q)\subset M\htens\tilde{M}$ is rigid and hence there exists a $t>0$ (which will be fixed for the rest of this proof) such that 
\begin{align}\label{ulighed-1} 
 \|(\id \tens \alpha_t)(\Delta(u))-\Delta(u)\|_2\leq \frac{\delta}{2} \ \text{ for all } u\in \U(Q).
 \end{align}
Define $S:=\{ h\in \Lambda \mid  \|\alpha_t(v_h)-v_h\|_2\leq \delta \}$. We need to show that $S$ contains a sequence whose Fourier coefficients (relative to $M=A\rtimes \Gamma$) all go to zero. To prove this it suffices to prove the following claim
\[
\forall \varps>0 \ \forall F\subset \Gamma \text{ finite }: S\cap \underbrace{\{h\in \Lambda \mid \|P_F(v_h)\|_2< \varps  \}}_{=:T_F}\neq \emptyset.
\]
(because letting $\varps$ run through the sequence $\frac1n$ and $F$ run through an increasing sequence exhausting $\Gamma$ we then get the desired sequence of unitaries.) Fix a finite subset $F\subset \Gamma$ and an  $\varps>0$. We first do a little auxiliary computation that will be needed in a short while. For $u=\sum_{h\in \Lambda} b_hv_h\in M$ we have 
\begin{align}\label{comp}
\left \| (\id \tens P_F)\Delta(u)\right \|_2^2& =\|(\id\tens P_F)\sum_{h\in \Lambda} b_hv_h\tens v_h \|_2^2 \notag\\
&= \sum_{h\in \Lambda} \|b_h\|_2^2\|P_F(v_h)\|_2^2\notag \\
&= \sum_{h\in \Lambda}  \|b_h\|_2^2\| P_F\sum_{g\in \Gamma}(v_h)_gu_g\|_2^2\notag \\
&= \sum_{h\in \Lambda} \|b_h\|_2^2\|\sum_{g\in F}(v_h)_gu_g\|_2^2\notag \\
&= \sum_{h\in \Lambda}\sum_{g\in F} \|b_h\|_2^2\|(v_h)_g\|_2^2 
\end{align}
By Lemma \ref{non-embedding-lem} we have $\Delta(Q)\nprec_{M\htens M} M\htens A$ so we can find  unitaries $u_n\in Q$ such that
\[
\|  E^{M\htens M}_{M\htens A}  (x \Delta(u_n) y  )\|_2\to 0 \ \text{ for all } x,y\in M\htens M.
\]
Putting $x=1$ and considering $y=1\tens u_g$ for some $g\in \Gamma$  we get
\begin{align*}
\|  E^{M\htens M}_{M\htens A}  (x\Delta(u_n) y  )\|_2 &= \|  (\id \tens E^{M}_{A})  (\Delta(u_n) (1\tens u_g)  )\|_2^2\\
&= \|  (\id \tens E^{M}_{A})  \sum_{h\in \Lambda} (u_n)_hv_h\tens v_hu_g \|_2^2\\
&=\|  \sum_{h\in \Lambda} (u_n)_hv_h\tens   E^{M}_{A} (v_hu_g )\|_2^2 \\
&=\|  \sum_{h\in \Lambda} (u_n)_hv_h\tens   (E^{M}_{A}  \sum_{\gamma \in \Gamma}(v_h)_\gamma  u_{\gamma g}  )\|_2^2 \\
&=\|  \sum_{h\in \Lambda} (u_n)_h v_h\tens   (v_h)_{g^{-1}} )\|_2^2\\
&=\sum _{h\in \Lambda} \| (u_n)_h \|_2^2 \|(v_{h})_{g^{-1}}\|_2^2.
\end{align*}
Since this quantity goes to zero for all $g\in \Gamma$ it follows from the computation \eqref{comp} that $\|(\id\tens P_F)\Delta(u_n)\|_2\to 0$. We can therefore find $u\in \U(Q)$ (more precisely one of the $u_n$'s) such that

\begin{align}\label{ulighed-2}
\|(\id\tens P_F)\Delta(u)\|_2<\frac{\varps}{2}.
\end{align}
Expanding $u$ relative to $M=B\rtimes \Lambda$ as $\sum_{h\in \Lambda}b_h v_h$  we get from this that
\[
\frac{\varps^2}{4}\geq \sum_{h\in \Lambda}\|b_h\|^2 \|P_F(v_h)\|_2^2\geq \sum_{h\in \Lambda \setminus T_F} \varps^2\|b_h\|_2^2,
\] 
and thus $\sum_{h\in \Lambda\setminus T_F}\|b_h\|_2^2\leq \frac14$.  Similarly,  using the inequality \eqref{ulighed-1} we get 
\begin{align}
\frac{\delta^2}{4} &\geq \|(\id\tens\alpha_t)\Delta(u)-\Delta(u) \|_2^2\notag\\
&= \| (\id\tens\alpha_t) \sum_{h\in \Lambda} b_hv_h\tens v_h -\sum_{h\in \Lambda} b_hv_h \tens v_h\|_2^2\notag\\
&= \| \sum_{h\in \Lambda} b_hv_h\tens (\alpha_t(v_h)-v_h)   \|_2^2\notag \\
&=\sum_{h\in \Lambda} \|b_hv_h\|_2^2\|\alpha_t(v_h)-v_h\|_2^2\tag{first leg orthogonal}\\
&\geq \sum_{h\in \Lambda\setminus S} \|b_h\|_2^2\|\alpha_t(v_h)-v_h\|_2^2 \tag{$v_h$ unitary}\\
&\geq \sum_{h\in \Lambda\setminus S} \|b_h\|_2^2\delta^2.
\end{align}
Hence  $\sum_{h\in \Lambda\setminus S}\|b_h\|_2^2\leq  \frac14$ so in total $\sum_{h\in \Lambda\setminus (S\cap T_F)}\|b_h\|^2\leqslant \frac12$. But since $\|u\|_2^2=1$ we therefore cannot have $S\cap T_F=\emptyset$.
\end{proof}
\section{Implementation from transferred rigidity}\label{implementation-section}
Throughout this section we consider again the group $\Gamma$ from Theorem \ref{chifan-peterson-thm}, the crossed product von Neumann algebra $M=A\rtimes \Gamma$ as well as the enlarged algebra $\tilde{M}$ and its deformation $\alpha_t$ introduced in section \ref{deform-from-cocycles}. Our goal  is to prove the following theorem which roughly states that if we  have a subalgebra $B\subset M$ with transferred rigidity (i.e.~the setup from Theorem \ref{transfer-thm}) then the deformation $\alpha_t$ is implemented by a partial isometry on this subalgebra. Precisely we prove the following.

\begin{thm}[Chifan-Peterson, \cite{chifan-peterson}]\label{implementation-thm}
Assume that $B\subset M$ is an abelian von Neumann subalgebra and that there exists a $t>0$ and a sequence $v_n\in \N_M(B)$\footnote{Recall that the \emph{normalizer} $\N_M(B)$ is defined as $\{u\in \U(M)\mid u^*Bu=B\}$.} such that 
\begin{itemize}
\item[(a)] $\|\alpha_t(v_n)-v_n\|_2<\delta:=\frac{1}{200}$ for all $n\in \NN$. 
\item[(b)] The Fourier coefficients of $v_n$ (relative to the decomposition $M=A\rtimes \Gamma$) go to zero in $2$-norm.
\end{itemize}
Then there exists a non-zero partial isometry $v\in \tilde{M}$ such that $\alpha_t(x)v=vx$ for all $x\in B$.
\end{thm}
\begin{rem}
In what follows we are going to refer to a partial isometry $v\in \tilde{M}$ such that $\alpha_t(x)v=vx$ for all $x\in B$ as  a \emph{partial isometry implementing $\alpha_t$ on $B$}. Note that if $v$ happens to be a unitary then conjugation by $v$ constitutes an honest implementation of the restriction of $\alpha_t$ to $B$.

\end{rem}
Define $\delta_t\colon M\to \tilde{M}\ominus M:=\{x\in \tilde{M}\mid E_M(x)=0\}$ by $\delta_t(x)=\alpha_t(x)-E_M(\alpha_t(x))$. For the estimates that are to come we first note that for a unitary $x\in M$ we have $\|\delta_t(x)\|_2\leq 1$ since
\begin{align}\label{eq117}
1=\|\alpha_t(x)\|_2^2= \|E_M(\alpha_t(x)) \|_2^2 +\| (1-E_M)\alpha_t(x)  \|_2^2 = \|E_M(\alpha_t(x))\|_2^2 +\|\delta_t(x)\|_2^2. 
\end{align}

We first need to prove  the following lemma.
\begin{lem}\label{c-lem}
Under the assumptions in Theorem \ref{implementation-thm} there exists $0<c<1$ such that $\|\delta_t(b)\|_2\leq c$ for all $b\in \U(B)$.
\end{lem}
\begin{proof}
The proof is divided into two cases.\\

\noindent \textbf{Case 1:} $\exists  b\in \U(B)$ such that $ \forall F\subset \Gamma$ finite:  $\limsup_n \|P_F(v_nbv_n^*) \|_2\leq \delta_0:=\frac16$.\\

\noindent(here the $v_n$'s are the ones we assume exist in Theorem \ref{implementation-thm}) In this case, we define $b_n=v_nbv_n^*$ and note that $b_n\in \U(B)$ since we assume $v_n\in \N_M(B)$. By Proposition \ref{monotone-prop} and assumption (a) we have $\|\alpha_s(v_n)-v_n\|_2< \delta$ for all $n\in \NN$ and all $s\in [0,t]$, so by replacing $t$ by a smaller number if necessary we may assume that $\|\alpha_t(b)-b\|_2<\delta$. 
Then  the assumption (a) implies that $\|\alpha_t(b_n)-b_n\|_2\leq 3\delta $ for all $n\in \NN$.
Now let $d\in \U(B)$ be given. Then
\begin{align}
\|\delta_t(d)\|_2^2 &=\ip{\delta_t(d)}{\delta_t(d) +E_M(\alpha_t(d))} \tag{$\delta_t(d)\perp M$}\\
&=\ip{\delta_t(d)  }{\alpha_t(d)}\notag\\
 &=\ip{\delta_t(d)  }{\alpha_t(b_nd b_n^*)} \tag{$B$ abelian}\\
&\leq |\ip{\delta_t(d)  }{(\alpha_t(b_n)-b_n) \alpha_t(d)\alpha_t (b_n^*)} | + |\ip{\delta_t(d)  }{b_n \alpha_t(d)\alpha_t (b_n^*)} | \notag\\
&\leq 3\delta  + |\ip{\delta_t(d)  }{b_n \alpha_t(d)\alpha_t (b_n^*)} | \notag\\
&\leq 6\delta  + |\ip{\delta_t(d)  }{b_n \alpha_t(d)b_n^*} | \tag{same trick on the right}\\ 
&= 6\delta  + |\tau_{\tilde{M}}(b_n \alpha_t(d^*)b_n^* \delta_t(d))|\notag \\
&= 6\delta  + |\tau_M \circ E_M (b_n \alpha_t(d^*)b_n^* \delta_t(d))| \notag  \\ 
&= 6\delta  + |\tau_M  (b_nE_M( \alpha_t(d^*)b_n^* \delta_t(d)))|\notag \\ 
&\leq 6\delta + \|E_M( \alpha_t(d^*)b_n^* \delta_t(d))\|_2\tag{Cauchy-Schwarz} 
\end{align}
By assumption 
\[
\limsup_n \|P_F(b_n^*) \|_2=\limsup_n\|P_{F^{-1}}(b_n)^*\|_2 =\limsup_n\|P_{F^{-1}}(b_n)\|_2   \leq \delta_0
\]
for all finite $F\subset \Gamma$,  and from Lemma \ref{lem-2} (with $x:=\alpha_t(d^*)\in (M)_1$ and $y:= \delta_t(d)\in (\tilde{M}\ominus M)_2$ (note the ``2''!) we therefore get
that $\limsup_n \|E_M( \alpha_t(d^*)b_n^* \delta_t(d))\|_2\leq 2\delta_0 $. In total we now have
\[
\|\delta_t(d)\|_2^2 \leq 6\delta + 2\delta_0=\frac{6}{200}+ \frac{2}{6}\simeq 0.3633<1.
\]
What is left is to treat the following situation:\\

\noindent \textbf{Case 2:} $\forall b\in \U(B) \ \exists F_b\subset \Gamma$ finite such that $\limsup_n\|P_{F_b}(v_nbv_n)\|_2> \delta_0.$ \\

\noindent Fix $b\in \U(B)$ and its corresponding\footnote{There may of course be several such sets, but we just choose one.} $F_b=:F$. Put again $b_n:=v_nbv_n^*$ and define
\[
\xi_n:=b_n-P_F(b_n) = \sum_{g\in \Gamma}(b_n)_gu_g -\sum_{g\in F}(b_n)_g u_g=\sum_{g\notin F} (b_n)_gu_g.
\]
Then
\[
\|\xi_n\|_2^2=\sum_{g\notin F}\|(b_n)_g\|_2^2= 1-\sum_{g\in F}\|(b_n)_g\|_2^2 = 1-\|P_F(b_n)  \|_2^2,
\]  
so by assumption $\liminf_n \|\xi_n\|_2\leq \sqrt{1-\delta_0^2}$. As $\delta_t(b)\perp M$ in $L^2(\tilde{M})$ we  have
\begin{align}
\|\delta_t(b)\|_2^2&= \ip{\delta_t(b)}{\delta_t(b) + E_M(\alpha_t(b))} \notag\\
&=\ip{\delta_t(b)}{\alpha_t(b)}\notag\\
&=\ip{v_n\delta_t(b) v_n^*}{v_n\alpha_t(b) v_n^*}\notag\\
&= | \ip{v_n\delta_t(b) v_n^*}{(v_n-\alpha_t(v_n))\alpha_t(b)v_n^*  + \alpha_t(v_n)\alpha_t(b)v_n^*} | \notag \\
&\leq \delta + | \ip{v_n\delta_t(b) v_n^*}{ \alpha_t(v_n)\alpha_t(b)v_n^*} | \tag{using (a) and $\|\delta_t(b)\|_2\leq 1$}\\
&\leq 2\delta + | \ip{v_n\delta_t(b) v_n^*}{ \alpha_t(v_n)\alpha_t(b)\alpha_t(v_n^*)} | \tag{same trick on the right}\\
&= 2\delta + | \ip{v_n\delta_t(b) v_n^*}{ \alpha_t(b_n)} |\tag{$b_n:=v_nbv_n^*$}\\
&= 2\delta + | \ip{v_n\delta_t(b) v_n^*}{ \alpha_t(\xi_n) +\alpha_t(P_F(b_n)) } |\tag{$\xi_n:=b_n-P_Fb_n$}\\
&\leq 2\delta + \| v_n\delta_t(b) v_n^*\|_2\|\alpha_t(\xi_n)\|_2  +| \ip{v_n\delta_t(b) v_n^*}{ \alpha_t(P_F(b_n)) } |\notag\\
&\leq 2\delta + \|\xi_n\|_2  +| \ip{v_n\delta_t(b) v_n^*}{ \alpha_t(P_F(b_n)) } |\notag\\
&=2\delta + \|\xi_n\|_2  +| \tau_{\tilde{M}}\left( \alpha_t(P_F(b_n))^*v_n\delta_t(b)v_n^*     \right)  |\notag\\
&=2\delta + \|\xi_n\|_2  +| \tau_{\tilde{M}}\left(   v_n^*\alpha_t(P_F(b_n))^*v_n\delta_t(b)     \right)  |\notag\\
&=2\delta + \|\xi_n\|_2  +| \tau_{M}\left( E_M(  v_n^*\alpha_t(P_F(b_n))^*v_n\delta_t(b)   )  \right)  |\notag\\
&=2\delta + \|\xi_n\|_2  +| \tau_M\left( v_n^* E_M( \alpha_t(P_F(b_n))^*v_n\delta_t(b)   )  \right)  |\notag\\
&\leq 2\delta + \|\xi_n\|_2  +\| E_M( \alpha_t(P_F(b_n))^* v_n\delta_t(b)   )  \|_2 \notag\\
&= 2\delta + \|\xi_n\|_2  +\| E_M( \alpha_t(P_{F^{-1}}(b_n^*)) v_n\delta_t(b)   )  \|_2 \tag{$P_F(x)^*=P_{F^{-1}}(x^*)$}
\end{align}
Since $\delta_t(b)\in \tilde{M}\ominus M$, Corollary \ref{cor-3} gives that last summand converges to zero. Thus
\[
\| \delta_t(b)\|_2^2\leq 2\delta +\liminf_n \|\xi_n\|_2 \leq 2\delta +\sqrt{1-\delta_0^2}=2 \frac{1}{200} + \sqrt{1-\frac{1}{6^2}}\simeq 0.9960 <1. 
\]
Thus $c:=2\delta +\sqrt{1-\delta_0^2}$ does the job in both case 1 and case 2.
\end{proof}

With Lemma \ref{funky-lem} at our disposal, the proof of Theorem \ref{implementation-thm} now becomes a triviality:
\begin{proof}[Proof of Theorem \ref{implementation-thm} ]
Lemma \ref{funky-lem} and equation \eqref{eq117} imply that $\|E_M(\alpha_t(x))\|\geq \sqrt{1-c^2}>0$ and  the desired conclusion therefore follows from  Corollary \ref{implement-cor}.
\end{proof}

\section{From partial implementation to unitary conjugacy}\label{conjugacy-section}
In this section we prove the uniqueness of Cartan subalgebras $B\subset M$ on which the deformation $\alpha_t\colon \tilde{M}\to \tilde{M}$ is implemented by a partial isometry. More precisely we prove the following:

\begin{thm}\label{blabla}
Let $B\subset M=A\rtimes \Gamma$ be Cartan and assume that there exists $t>0$ and a non-zero partial isometry $w\in \tilde{M}$ such that $\alpha_t(x)w=wx$ for all $x\in B$. Then there exists $u\in \U(M)$ such that $B=uAu^*$.
\end{thm}
First we need a little result which shows that absence of virtual conjugacy provides us with control over the relative commutants.

\begin{lem}[Popa]\label{rel-commutant-lem}
If $B\subset M=A\rtimes \Gamma$ is von Neumann subalgebra and $B\nprec_M A$ then $B'\cap \tilde{M}\subset M$.
\end{lem}
Note, in particular, that if $B$ is a MASA in $M$ then the lemma shows that if $B\nprec_M A$ then the relative commutant $B'\cap \tilde{M}$ is just $B$. 
\begin{proof}
Since $B\nprec_M A$ we can find a sequence of unitaries $v_n\in B$ with Fourier coefficients going to zero in 2-norm. Let $x\in \U(B'\cap \tilde{M})$ be given and put $y:= x-E_M(x) \in \tilde{M}\ominus M$. We need to show that $x\in M$. Since $x$ and $v_n$ commute we have
\[
1=\|v_n\|_2=\|E_M(v_n)\|_2=\|E_M(x^*v_nx)\|_2 \leq \|E_M(x^*v_n y)\|_2+\|E_M(x^*v_nE_M(x))\|_2.
\]
Note that by Lemma \ref{lem-1} the first term converges to zero. Using the Schwarz inequality for u.c.p.~maps on the conditional expectation $E_M$ we get
\begin{align*}
\|E_M(x^*v_nE_M(x))\|_2^2 &=\tau_M((E_M(x^*v_nE_M(x)))^*E_M(x^*v_nE_M(x)))\\
&\leq \tau_M\circ E_M((x^*v_nE_M(x))^*(x^*v_nE_M(x))) \\
&=\tau_{\tilde{M}}(E_M(x)^*E_M(x))\\
&=\|E_M(x)\|_2^2.
\end{align*}  
Hence  $ 1= \|x\|_2\leq \|E_M(x)\|_2 $ and therefore
\[
1=\|x-E_M(x) +E_M(x)\|_2^2 =\|x-E_M(x)\|_2^2 +\|E_M(x)\|_2^2 \geq \|x-E_M(x)\|_2^2 +1.
\]
This shows $E_M(x)=x$ and thus $x\in M$.  

\end{proof}

Furthermore we need the following lemma.

\begin{lem}\label{lem-63} 

Assume the setup in Theorem \ref{blabla} and assume furthermore $B\nprec_M A$. If for all $x\in B$ we have $\alpha_t(x)w_1=w_1x$ and $\alpha_t(x)w_2=w_2x$ for some elements $w_1,w_2\in \tilde{M}$  then $w_2^*w_1\in B$ and $w_2w_1^*\in \alpha_t(B)$.
\end{lem}
\begin{proof}
Let $x\in B$ be given.
\[
(w_2^*w_1)x=w_2^*\alpha_t(x)w_1=(\alpha_t(x^*)w_2)^*w_1=(w_2x^*)^*w_1= x(w_2^*w_1).
\]
Hence $w_2^*w_1$ is in $B'\cap\tilde{M}$ which, by Lemma \ref{rel-commutant-lem}, equals $B$ since $B$ is Cartan. The other statement is proven similarly.
\end{proof}

Lastly we will  need the following little result of more general nature.
\begin{lem}\label{p-lig-en-lem}
Let $M$ and $N$ be $\twoone$-factors included in a common finite von Neumann algebra $\tilde{M}$ and assume that there exists a projection $p\in N$ and a unitary $x\in \tilde{M}$ such that $x (pNp)x^*\subseteq M$. Then there exists a unitary $z\in \tilde{M}$ such that $zNz^*\subseteq M$
\end{lem}
\begin{proof}
By replacing $M$ with $x^*Mx$ we may assume that $pNp\subseteq M$.  Choose an integer $n$ such that $\frac1n\leq \tau(p)$ and decompose $1$ as the orthogonal sum $p_1+\dots + p_n$ of projections in $N$ that are all equivalent to a fixed subprojection $p_0\leq p$ of trace $\frac1n$. That is,  there exists $v_1,\dots, v_n\in N$ with $ v_i^*v_i^{\phantom{*}}=p_i$ and $v_i^{\phantom{*}}v_i^*=p_0$. Now we do the same thing in $M$: choose orthogonal projections $q_1,\dots, q_n\in M$ and partial isometries $w_1,\dots, w_n\in M$ such that
$q_1+\cdots +q_n=1$, $w_i^*w_i^{\phantom{*}}=q_i$ and $w_i^{\phantom{*}}w_i^*=p_0$. Then putting $z:=\sum_{i=1}^n w_i^*v_i^{\phantom{*}} \in \tilde{M}$ we get, for $x\in N$, that
\[
z x z^* =\sum_{i,j=1 }^n  w_i^* (v_i^{\phantom{*}}x v_j^*)w_j^{\phantom{*}} = \sum_{i,j=1 }^n  w_i^* \underbrace{(pv_i^{\phantom{*}}x v_j^*p)}_{\in pNp\subseteq M }w_j\in M ,
\]
so that $zNz^*\subseteq M$. Furthermore, since the source spaces of the $v_i$'s are orthogonal we get
\[
zz^*=\sum_{i,j=1}^n w_i^*v_i^{\phantom{*}}v_j^*w_j^{\phantom{*}}=\sum_{i=1}^nw_i^* v_i^{\phantom{*}}v_i^* w_i^{\phantom{*}}=\sum_{i=1}^n  w_i^* w_i^{\phantom{*}} w_i^*w_i^{\phantom{*}}=\sum_{i=1}^n q_i=1, 
\]
and similarly $z^*z=1$. 
\end{proof}
We are now equipped to prove the main result of this section.
\begin{proof}[Proof of Theorem \ref{blabla}]
Assume that $A$ and $B$ are not conjugate. Since they are both Cartan in $M$ this implies (by Corollary \ref{cartan-virtual})  that $B\nprec_M A$ and (by Lemma \ref{rel-commutant-lem}) therefore $B'\cap \tilde{M}\subset M$. Since $B\subset M$ is Cartan it is in particular a MASA in $M$ and thus $B'\cap\tilde{M}=B$. Now, let $u\in \N_M(B)$ be given. For $x\in B$ we have
\begin{align*}
\alpha_t (u^*x u)w=w(u^*xu) 
\end{align*}
 and rearranging this equality we get $\alpha_t(x)(\alpha_t(u)wu^*)=(\alpha_t(u)wu^*)x$. Thus $\alpha_t(u)wu^*$ also implements
$\alpha_t$ on $B$ and by Lemma \ref{lem-63}  we therefore get $w^*(\alpha_t(u)wu^*)\in B$. Thus
\[
w^* \alpha_t(u)w^*\in Bu\subset M.
\]
Since $u\in \N_M(B)$ was arbitrary and $B$ is Cartan this implies $w^*\alpha_t(M)w\subseteq M$. From Lemma \ref{lem-63} we also get  $ww^*\in \alpha_t(B)$ and hence there exists a projection $p\in B$ such that $ww^*=\alpha_t(p)$. Extending $w$ to a unitary $z\in \tilde{M}$ with $w=\alpha_t(p)z$ \footnote{Since $\tilde{M}$ is  finite $1-ww^*$ and $1-w^*w$ are also equivalent \cite[Propositio V.1.38]{takesaki-1} and therefore this is possible.} we now get $z^*\alpha_t(pMp)z =w^*\alpha_t(M)w\subseteq M$. Applying Lemma \ref{p-lig-en-lem} we may assume that $p=1$.
Then we have $z^*\alpha_t(M)z\subseteq M$ and we now show that this cannot be the case by proving the following:\\

\noindent \textbf{Claim:} If $g_n\in \Gamma$ and $\|b(g_n)\| \To \infty$ then $\| E_M(z_1\alpha_t(u_{g_n})z_2)\|_2\To 0$ for all $z_1,z_2\in \tilde{M}$.\\

\noindent  Given the claim the desired contradiction easily follows: Since the cocycle $b\colon \Gamma \to H$ is assumed unbounded we can find a sequence $g_n\in \Gamma$ such that $\|b(g_n)\|\to \infty$ and since $z^*\alpha_t(M)z\subseteq M$ we have
\[
1= \|z^*\alpha_t(u_{g_n})z\|_2 =\|E_M(z^*\alpha_t(u_{g_n})z)\|_2\
\]
contradicting the claim. 
\begin{proof}[Proof of claim]
By the $M=(1\tens A)\rtimes \Gamma$-linearity of $E_M$ it suffices to look at $z_1,z_2\in D\tens 1\subset (D\htens A)\rtimes \Gamma=:\tilde{M}$. Since elements of the form $\omega(\xi)\tens 1$ spans a weakly dense $*$-subalgebra in $D\tens 1$ it furthermore suffices to treat the case when $z_1=\omega(\xi)\tens 1$ and $z_2=\omega(\eta)\tens 1$. Then we have
\begin{align*}
z_1\alpha_t(u_{g_n})z_2 &= (\omega(\xi) \tens 1) (\omega(tb(g_n))\tens 1) u_{g_n}(\omega(\eta)\tens 1)\\
&=(\omega(\xi +t b(g_n) +\pi_{g_n}(\eta))\tens 1)u_{g_n},
\end{align*}
and hence
\[
\|E_M(z_1\alpha_t(u_{g_n}))z_2\|_2= \exp(-\|\xi +tb(g_n) + \pi_{g_n} (\eta)\|^2 ).
\]
Since the sequence $\pi_{g_n}(\eta)$ is bounded in 2-norm  and  $\lim_n\|b(g_n)\|=\infty$ the last expression goes to zero and the claim follows.
\end{proof}
This concludes the proof of Theorem \ref{blabla}.
\end{proof}
\begin{rem}
Although hidden in the proof of Theorem \ref{blabla}, the claim above shows exactly why the unboundedness of the cocycle $b$ is important. Heuristically, it states that the unitaries in $M$ arising from a sequence in $\Gamma$ exhibiting the unboundedness of $b$ are being moved  more and more towards the orthogonal complement of $M$ in $\tilde{M}$ by the deformation $\alpha_t$ --- even after conjugation by elements from $\tilde{M}$! This is the underlying reason why we cannot conjugate $\alpha_t(M)$ back into $M$ which, as we just saw, leads to the desired contradiction in the proof of Theorem \ref{blabla}.

\end{rem}

\section{The grand finale}
To get the main result we now just have to assemble the components obtained in the previous sections. Recall that we are after  the following result.

\begin{thm*}[Chifan-Peterson, \cite{chifan-peterson}]
Assume that $\Gamma$ is a discrete group acting ergodically, essentially freely and measure preserving on a standard probability space $(X,\mu)$ and put $M=L^\infty(X)\rtimes \Gamma$. Assume furthermore that $\Lambda$ is another (essentially free, ergodic p.m.p.) action on another probability space $(Y,\nu)$  such that $M=L^\infty(Y)\rtimes \Lambda$. If
\begin{itemize}
\item there exists an infinite subgroup $\Gamma_0< \Gamma$ with property $\T$ and
\item a representation $\pi\colon \Gamma\to O(H)$ on a real Hilbert space $H$ with an unbounded 1-cocycle $b\colon \Gamma \to H$,
\end{itemize}
then there exists a unitary $u\in M$ such that $uL^{\infty}(X)u^*=L^\infty(Y)$. 
\end{thm*}
\begin{proof}
From the assumptions we obtain (as in Section \ref{deform-from-cocycles}) the  von Neumann algebra $\tilde{M}$ containing $M$ and the family of automorphisms $\alpha_t$ on it. Furthermore, by Theorem \ref{transfer-thm} the assumptions in Theorem \ref{implementation-thm} are fulfilled and we therefore get a partial isometry $w\in \tilde{M}$ such that $\alpha_t(x)w=wx$ for all $x\in B:=L^\infty(Y)$.  Applying Theorem \ref{blabla} now gives that $L^\infty(X)$ and $L^{\infty}(Y)$ are unitarily conjugate inside $M$.

\end{proof}

\def\cprime{$'$}

\end{document}